\documentclass{amsart}
\usepackage{graphicx}
\graphicspath{ {figuras/} }
\usepackage[utf8]{inputenc}
\usepackage{amsmath}
\usepackage{algorithm}
\usepackage{algorithmic}
\usepackage{amsfonts}
\usepackage{amssymb}
\usepackage{graphicx}
\usepackage{xcolor, colortbl}
\usepackage{lscape}
\usepackage{array, multirow, multicol}

\newtheorem{theorem}{Theorem}[section]
\newtheorem{lemma}[theorem]{Lemma}

\theoremstyle{definition}
\newtheorem{definition}[theorem]{Definition}
\newtheorem{example}[theorem]{Example}

\theoremstyle{remark}
\newtheorem{remark}[theorem]{Remark}

\newtheorem{corollary}[theorem]{Corollary}

\numberwithin{equation}{section}

\begin{document}

\title{Is it possible that the Goldbach's and Twins primes conjectures are true with an algebraic approach?}

\author{Juan Carlos Riano-Rojas}
\address{Department of Mathematics, Universidad Nacional de Colombia, Caldas, Manizales}
\curraddr{Departamento de Matemáticas y Estadística,
Universidad Nacional de Colombia, Caldas, Manizales.}
\email{jcrianoro@unal.edu.co}

%\thanks{To Gonzalo Mediana Arellano for his support in LaTeX}

%\author{Author Two}
%\address{Mathematical Research Section, School of Mathematical Sciences,
%Australian National University, Canberra ACT 2601, Australia}

%    General info
%
%\date{January 1, 2001 and, in revised form, June 22, 2001.}
%
\dedicatory{This paper is dedicated to my wife Elisabeth and my daugthers Tata and Jana, and in memory of my friend Omar }


\begin{abstract}
In this paper, using an algebraic approach, it is intended to show that the Goldbach's and Twin primes conjectures are true, building, for each $m>2$, an isomorphism between posets. One of the posets is the set of coprimes less than $m$, while the other is endowed with an operation that grants it an abelian group structure. Special features of this operation are demonstrated in the document, which allow characterizing the even numbers, as if they were their fingerprint; furthermore, such an operation locates, in a natural way, the pairs that satisfy the conjecture. Moreover, an algorithm that generates pairs of numbers that satisfy the conjecture is presented, and examples of some of the beautiful symmetries of the orbits of cyclic subgroups are shown, for the proposed abelian group.
\end{abstract}

\maketitle

%\section*{This is an unnumbered first-level section head}
%This is an example of an unnumbered first-level heading.
%
%%% The correct journal style for \specialsection is all uppercase; a known bug
%%% in amsart.cls prevents this, so input must be uppercase until it is fixed.
%%\specialsection*{This is a Special Section Head}
%\specialsection*{THIS IS A SPECIAL SECTION HEAD}
%This is an example of a special section head%
%%%%%%%%%%%%%%%%%%%%%%%%%%%%%%%%%%%%%%%%%%%%%%%%%%%%%%%%%%%%%%%%%%%%%%%%%
%\footnote{Here is an example of a footnote. Notice that this footnote
%text is running on so that it can stand as an example of how a footnote
%with separate paragraphs should be written.
%\par
%And here is the beginning of the second paragraph.}%
%%%%%%%%%%%%%%%%%%%%%%%%%%%%%%%%%%%%%%%%%%%%%%%%%%%%%%%%%%%%%%%%%%%%%%%%%
%.

\section{Introducción}

Goldbach's conjecture has been considered one of the great problems, still open, of mathematics, according to G.H. Hardy, in 1921. The interest of scientifics on this conjecture can be verified by several investigations, which show results aimed at achieving boundaries of numbers that meet the conjecture. Other articles transform the problem, seeking to represent some class of numbers, trying to reduce the boundaries that guarantee these new representations; for instance, Sinisalo \cite{Sinisalo} reported a study checking this conjecture up to $4*10^{11}$ by the IBM 3083 mainframe with a vector processor; however, these authors did not develop an explicit proof. They only tested the conjecture. In \cite{Wu}, Wu J. Tried to give a more comprehensive treatment of Chen's double sieve and prove an upper boundary, sharper than $D(N) \leq 7.8209 \Theta(N)$.\\

In \cite{Tao}, Tao proved that, every odd number $N$, greater than $1$ can be expressed as the sum of at most five primes, improving the result of Ramaré, where every even natural number can be expressed as the sum of at most six primes. They follow the circle method of Hardy-Littlewood and Vinogradov, together with Vaughan's identity; their additional techniques, which may be of interest for other Goldbach-type problems, include the use of smoothed exponential sums and optimization of the Vaughan identity parameters to save or reduce some logarithmic losses. Despite their good approximation, these authors do not use the construction of coprimes, nor the algebraic structure to prove that the conjecture holds for all pairs greater than four. \\

In \cite{Zhao} Zhao, proven that, every sufficiently large even integer can be represented as a sum of four squares of primes and $46$ powers of $2$, but they do not directly use the coprimes and again, they only establish a bound that exceeds the existing bound in previous works. In \cite{Ren}-\cite{Garaev}, it is investigated the distribution of values of the primes plus powers of two, seeking to represent natural numbers to improve the bounds that achieve such a representation. In \cite{Helfgott}, Helfgott prensented a proof of the ternary Goldbach conjecture for all $n\geq C = 10^{27}$, follow an approach based on the circle method, the
large sieve and exponential sums. Some ideas coming from Hardy, Littlewood and Vinogradov are reinterpreted from a modern perspective. While all work has to be explicit, the focus is on qualitative gains. Nevertheless, they do not prove the strong Goldbach conjecture. They neither use the algebraic nature of the coprimes, nor do they generalize the multiplication of coprimes, over $\mathbb{Z}$, as it is carried out in this research.

The Twins Primes Conjecture is one of the oldest in number theory. Many mathematicians continue to research and develop new techniques and approaches to try to address this conjecture. Tools such as sieve theory, harmonic analysis, and analytical methods have been used, but no definitive proof has been found so far.

In \cite{Wang}, presented a formal proof of the twin prime conjecture based on a novel mathematical model of two dimensional mirror primes $\mathbb{P}_\mu \subset \mathbb{P}\times \mathbb{P}$, and their symmetric properties, but the construction differs from that of this work, since it does not use the isomorphism between poset proposed.

In this paper, the Goldbach's conjecture and the Twins primes conjecture are proven by building, for each $m>2$, an isomorphism between posets. One of the posets is the set of coprimes less than $m$, while the other is endowed with an operation that grants it an abelian group structure. Certain special characteristics of this operation are demonstrated, which serve to locate the pairs that satisfy the conjecture. This work follows the next distribution: in the second section the basic theoretical notions of the coprime and non-coprime numbers generalized in $\mathbb{Z}$ are presented.  

In the third section the coprimes that are used to construct Euler's $\phi$ function are defined. An algorithm that generates pairs of numbers that satisfy the conjecture is presented.

In the fourth section be show the figures obtained from the cyclical subgroups of the group proposed in this work are included, which reflect underlying symmetries and regularities patterns, showing the implicit harmony in the coprimes that can be studied in future works. Furthermore, for each $m > 2$, an isomorphism between the posets is defined, in order to guarantee the proof of the central theorem. In the fifth section, the Goldbach's conjecture is proven. In the sixth section the twin primes conjecture is proved. 

Finally, The most relevant conclusions of this research are presented.

\section{theoretical development}

\subsection{Primes and Coprimes in $\mathbb{N}$}

 Primes and Coprimes: In arithmetic, Euler was one of the great mathematicians who saw the important relationship that prime numbers have in Mathematics. Euler proposed the function $\phi (n)$ that counts the number of coprimes or relative primes less than a natural number $m$. This function is a key to the development of this work. For this reason, a theoretical framework, necessary for the proof of the main result of this research, is presented below.

\begin{definition}\label{def1}
Let $\langle \mathbb{N},+,\cdot,|,0,1 \rangle $, be the set of natural numbers, with the operations sum, product and the conventional divisibility relation. Let $x , y \in \mathbb{N}$, be denoted by $ x \wedge y $, the greatest common divisor of $x$ and $y$;  $x$ and $y$ are said to be coprimes (or relatively prime) if and only if $x \wedge y = 1$

Let $m \in \mathbb{N}$ be a fixed number. 
The set of all coprimes less than $m$, will be denoted by $C(m)$, to the set of all coprime numbers whith number $m$, that will be noted by $C(m)$, which is explicitly stated as: 

\[ C(m) = \lbrace s \in \mathbb{N} : m \wedge s = 1 \rbrace.\] In a complementary way, the set of non-coprime naturals with the number $m$, is constructed. This set will be denoted by $\lambda(m)$, that can be written as: \[ \lambda(m) = \lbrace s \in \mathbb{N} : m \wedge s \neq 1 \rbrace\] 
\end{definition}

Comment \ref{remar1} indicates the basic notation and a property of number theory, which will be used to prove some theorems.

\begin{remark}\label{remar1}

The ideas discussed in this document use the following conventions:

\begin{enumerate}
\item $\mathbb{N}^{*}$, are the nonzero natural numbers.

\item $\mathbb{P}$, represents the set of prime numbers.

\item Given a natural number $m$ and $x \in \mathbb{N}, x \leq m $, it is said that, $m - x$, is the relative complement of $x$, with respect to $m$; then it will be called the co-opposite of $x$.
\end{enumerate}

\end{remark}

In example \ref{ejemp1}, we recall the notions of coprimes and non-coprimes, for a number $m$.

\begin{example}\label{ejemp1}
Consider $m = 6$, its decomposition into prime factors $m = 2\cdot3 $, allows to identify the coprimes as the natural ones that, when decomposed into prime factors, none of their powers is divisible by $2$ or $3$, this is:  \[C(6) = \lbrace s \in \mathbb{N} : 6 \wedge s = 1 \rbrace\] 
\[C(6) = \lbrace 1,5,7,11,13,17,19,23,29,31,35,37,41,43,...,55, ..., 65,... \rbrace. \] Another way of looking at it is: \[C(6) = \lbrace s \in \mathbb{N} : s = \prod^{\nu}_{j=0}p_{j}^{n_j}, p_{j} \in \mathbb{P}, p_{j} \neq 2, p_{j} \neq 3, n_j, \nu \in \mathbb{N}. \rbrace\]

Figure \ref{fig1}  shows the tree of natural numbers connected by the order relationship generated by divisibility. A partition of the two blue and magenta sets shows coprime and non-coprime numbers of the number $6$, respectively. In the second level of the tree, the primes $2,3,5,7,11...$, can be observed; the third level is made up of products of pairs of primes belonging  to the second level; in the fourth level, products of triples of primes are shown, etc. Edges of the same color show the divisibility relationship between elements of the same set, while gray edges represent non-coprime numbers, which have some coprime factor; For example, the number $75$ has factors of the coprime and non-coprime set.
\end{example}

\begin{figure}[tb]
\includegraphics[scale=0.075]{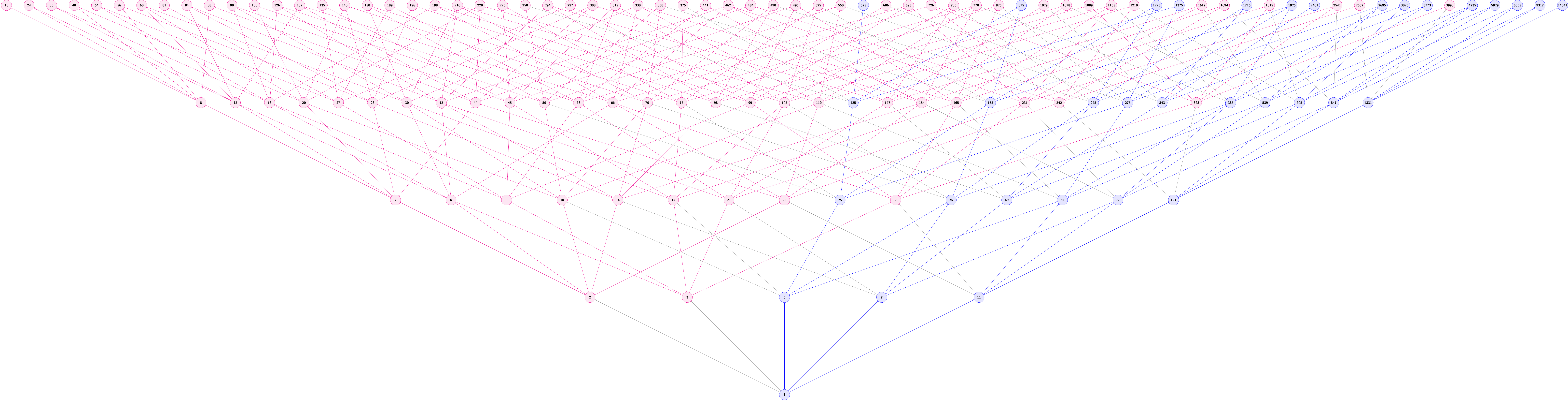}
\caption{$C(6)$ in blue and $\lambda(6)$ in magenta}\label{fig1}
\end{figure}

\subsection{Primes and Coprimes in $\mathbb{Z}$}

The notion of coprimes and non-coprimes is generalized to $\mathbb{Z}$, seeking to build one of the required posets.

\begin{definition}\label{def2}
Let $\langle \mathbb{Z},+,\cdot,|,0,1 \rangle $, be the set of integer numbers, with the operations sum, product and the conventional divisibility relation. Let $x , y \in \mathbb{Z}$, be denoted by $ x \wedge y $, the greatest common divisor of $x$ and $y$;  $x$ and $y$ are said to be coprimes (or relatively prime) if and only if $x \wedge y = 1$. 

The properties of arithmetic that help to characterize the problem are based on a well-known property that is recalled in the comment \ref{remar2}.

\begin{remark}\label{remar2}
Let $ x, y \in \mathbb{Z}^{*}$, then, it is had that $ x \wedge y = d \Leftrightarrow (\exists s, t \in \mathbb{Z}) ( xs + yt = d )$ (greatest common divisor property)\label{ecua1}

\end{remark}

Let $m \in \mathbb{Z}$ be fixed, to the set of all coprime numbers whith number $m$, that will be noted by $\mathbb{C}(m)$, which is explicitly stated as: 

\[\mathbb{C}(m) = \lbrace s \in \mathbb{Z} : m \wedge s = 1 \rbrace.\] In a complementary way, the set of non-coprime naturals with the number $m$, is constructed. This set will be denoted by $\Lambda(m)$, that can be written as: \[ \Lambda(m) = \lbrace s \in \mathbb{Z} : m \wedge s \neq 1 \rbrace\]

\end{definition}
The notion of congruence modulo $m$, can be restricted to $\mathbb{C}(m)$, as follows

\begin{definition}\label{def3}
Let $x,y \in \mathbb{C}(m)$.  It is said that, 
\[ x \equiv_{\mathbb{C}(m)} y  \Leftrightarrow \exists k  \in \mathbb{Z} \left(x - y = mk\right)\]
\end{definition}

In Appendix \ref{appendix}, the proof of the following result is performed.

\begin{theorem}\label{teo1}
For each $m>1$, it is had that  $\equiv_{\mathbb{C}(m)}$, is an equivalence relation on $\mathbb{C}(m)$
\end{theorem}

The quotient set is built in the traditional way, with the equivalence relation $\equiv_{\mathbb{C}(m)}$.

\begin{definition}\label{def4}
Let $a \in \mathbb{C}(m)$. The equivalence class of $a$, for the relationship $\equiv_{\mathbb{C}(m)}$ is built as:  
\[ \overline{a} = \left\lbrace z \in \mathbb{C}(m) : z \equiv_{\mathbb{C}(m)} a \right\rbrace  \]
The set of equivalence classes is called the quotient set and it will denote it by  \[\mathbb{C}(m)/\equiv_{\mathbb{C}(m)}=\left\lbrace \overline{z}: z \in \mathbb{C}(m) \right\rbrace \]
\end{definition}

\begin{remark}\label{remar3}
The cardinal of the quotient set coincides with the number of coprimes less than $m$, which is the value of the function $\phi(m)$, which was proposed by Euler. This quotient set is also noted by $\Gamma(m)=\mathbb{C}(m)/\equiv_{\mathbb{C}(m)}$. This will be used in section \ref{secc3}, in more detail.
\end{remark}

The multiplication of $\cdot$ de $\mathbb{Z}$,  is restricted to $\mathbb{C}(m)$, o endow it with an algebraic structure in the natural way, as follows:

\begin{definition}\label{def5}
Let $x,y \in \mathbb{C}(m)$.  The product in $\mathbb{C}(m)$, is defined as:  
\[ \cdot:\mathbb{C}(m)\times \mathbb{C}(m) \longrightarrow \mathbb{C}(m)\]
\[ \; \; \; (x , y) \longmapsto \cdot(x,y)= x\cdot y\]
\end{definition}

All product properties are inherited to $\mathbb{C}(m)$. In appendix  \ref{appendix}, only the closing property is demonstrated, as it is the least obvious.

\begin{theorem}\label{teo2}
For each $m>1$, the $\cdot$ is an operation on $\mathbb{C}(m)$, which is closure, modulative, associative and commutative
\end{theorem}

The $\equiv_{\mathbb{C}(m)}$, is compatible with the product $\cdot$ over $\mathbb{C}(m)$. In appendix \ref{appendix} the proof of the compatibility of the relation is included, with the aim of providing the quotient set $\mathbb{C}(m)/ \equiv_{\mathbb{C}(m)}$, with an algebraic structure.

\begin{theorem}\label{teo3}
For each $m>1$, it is had that \\ 
$\forall x,y,u,v \in \mathbb{C}(m) \left(  x \equiv_{\mathbb{C}(m)} y \wedge u \equiv_{\mathbb{C}(m)} v \Rightarrow x \cdot u \equiv_{\mathbb{C}(m)} y\cdot v \right) $ 
\end{theorem} 

This allows us to extend the product in $\mathbb{C}(m)$, to the quotient set $\mathbb{C}(m)/\equiv_{\mathbb{C}(m)}$ as follows

\begin{definition}\label{def6}
Let $\overline{a},\overline{b} \in \mathbb{C}(m)/\equiv_{\mathbb{C}(m)}$, the product $\odot$ is defined as:  
\[ \odot :\mathbb{C}(m)/\equiv_{\mathbb{C}(m)} \times \mathbb{C}(m)/\equiv_{\mathbb{C}(m)} \longrightarrow \mathbb{C}(m)/\equiv_{\mathbb{C}(m)}\]
\[ \; \; \; \; \; (\overline{a} , \overline{b}) \longmapsto \odot(\overline{a},\overline{b})= \overline{a}\odot\overline{b}=\overline{a\cdot b} \]
\end{definition}

This multiplication $\odot$, gives algebraic structure to the set quotient.

\begin{theorem}\label{teo4}
For each $m>1$, it is had that  $\left\langle \mathbb{C}(m)/\equiv_{\mathbb{C}(m)},\odot,\overline{1} \right\rangle $, is an abelian group.
\end{theorem}

\begin{proof}
Only the invertive property will be proved, since associativity, commutative and modulative are inherited from  $\mathbb{C}(m)$. Let $\overline{x} \in \mathbb{C}(m)/\equiv_{\mathbb{C}(m)}$, be arbitrary, such that:
\begin{enumerate}
\item $x \wedge m = 1$, by definition of $x \in \mathbb{C}(m)$.

\item $\exists s , t \in \mathbb{Z} \left(  x\cdot s + m\cdot t = 1 \right) $ is had by \ref{ecua1}.

\item $\exists s , t \in \mathbb{Z} \left( x\cdot s - 1 = m\cdot (-t)\right) $, solving for $m\cdot t$ on the previous line.

\item $\exists s \in \mathbb{Z} \left( x \cdot s \equiv_{\mathbb{C}(m)} 1 \right) $, by definition of $\equiv_{\mathbb{C}(m)}$

\item $\exists s \in \mathbb{C}(m) \left( \overline{x\cdot s}=\overline{1}\right) $,  since the classes are the same for related elements.

\item $\exists s \in \mathbb{C}(m) \left( \overline{x}\odot \overline{s}=\overline{1} \right) $, for compatibility.

\end{enumerate}

The last one  is equivalent to the fact that $x$ has a multiplicative inverse

\end{proof}

In appendix \ref{appendix}, it is shown that the co-opposites of module $\overline{1}$ are nilpotent.

\begin{theorem}\label{teo5}
For each $m>1$, in $\left\langle \mathbb{C}(m)/\equiv_{\mathbb{C}(m)},\odot,\overline{1} \right\rangle $, it is had that $\overline{m-1}\odot\overline{m-1}=\overline{1}$.
\end{theorem}

With the usual order of the natural numbers, between $1$ and $m$, in definition \ref{def7}, it is intended to provide the quotient set, with a structure of total order.

\begin{definition}\label{def7}
Let $\overline{x},\overline{y} \in \mathbb{C}(m)/\equiv_{\mathbb{C}(m)}$. We will say that $\overline{x} \preceq \overline{y}$, if only if, $\exists s, t \in \left\lbrace 1,2,...,m-1\right\rbrace \left(s\leq t \wedge s\equiv_{\mathbb{C}(m)} x \wedge t\equiv_{\mathbb{C}(m)} y \right)$.
\end{definition}

The quotient set has the structure of poset, see proof in appendix \ref{appendix}.

\begin{theorem}\label{teo6}
For each $m>1$, it follows that $\left\langle \mathbb{C}(m)/\equiv_{\mathbb{C}(m)},\preceq \right\rangle $ is a poset. Also, $\preceq$ is total order.
\end{theorem}

The following property shows an important characteristic that relates the module $\overline{1}$, the co-opposite to the module $\overline{m-1}$, the operation $\odot$ and the order $\preceq$ previously defined.

\begin{theorem}\label{teo7}
For each $m>1$, and for each $\overline{x} \in \mathbb{C}(m)/\equiv_{\mathbb{C}(m)}$, if $\overline{1} \preceq \overline{x} \preceq \overline{m-1}$ then, when operating by $\odot$ on the inequality, each term by $\overline{m-1}$  (the co-opposite of $\overline{1}$) has that  $\overline{m - 1}\odot \overline{1} \succeq \overline{m - 1}\odot \overline{x} \succeq \overline{m - 1}\odot\overline{m-1}$, which is equivalent to that  $\overline{1} \preceq \overline{m - x} \preceq \overline{m-1}$.
\end{theorem}

\begin{proof}
Let $\overline{x} \in \mathbb{C}(m)/\equiv_{\mathbb{C}(m)}$, such that  $\overline{1} \preceq \overline{x} \preceq \overline{m-1}$:
\begin{enumerate} 
\item $\overline{m - 1}\odot \overline{1} = \overline{m -1}$, by modulative property.

\item $\overline{m - 1}\odot\overline{x} = \overline{(m - 1)\cdot x} = \overline{m\cdot x - x}= \overline{m - x}$, by compatibility and distributive property.

\item $\overline{m - 1}\odot\overline{m-1} = \overline{1}$, by the idempotency property of the co-opposite.
\end{enumerate}

It only remains to show that the order is preserved, for this without loss of generality, taking $1\leq x \leq m-1$, the usual order in integers, and multiplying by $-1$, then adding $m$ to each term, we obtain $m-1\geq m - x \geq 1$. Finally, the inequality is given, by applying the equivalence classes to each member of the inequality, this is $\overline{m - 1}\succeq \overline{m - x} \succeq \overline{1}$.

\end{proof}

\begin{theorem}\label{teo8}
For each $m>1$, and for each $\overline{x} \in \mathbb{C}(2m)/\equiv_{\mathbb{C}(2m)}$, if $\overline{1} \preceq \overline{x} \preceq \overline{m}$ then, by operating $\odot$ to the inequality each term times $\overline{2m-1}$ (the co-opposite of $\overline{1}$ in $\mathbb{C}(2m)/\equiv_{\mathbb{C}(2m)}$), we have that $\overline{2m - 1}\odot \overline{1} \succeq \overline{2m - 1}\odot \overline{x} \succeq \overline{2m - 1}\odot\overline{m}$, which is equal to $\overline{m} \preceq \overline{2m - x} \preceq \overline{ 2m-1}$.

\end{theorem}

\begin{proof}
Let $\overline{x} \in \mathbb{C}(2m)/\equiv_{\mathbb{C}(2m)}$, such that $\overline{1} \preceq \overline{x} \preceq \overline{m}$:
\begin{enumerate} 
\item $\overline{2m - 1}\odot \overline{1} = \overline{2m -1}$, by modulative property.

\item $\overline{2m - 1}\odot\overline{x} = \overline{(2m - 1)\cdot x} = \overline{2m\cdot x - x}= \overline{2m - x}$, for compatibility and distributive properties in $\mathbb{Z}$.

\item $\overline{2m - 1}\odot\overline{m} = \overline{2m^2 - m} =\overline{2m - m} = \overline{m} $, by  distributive property and $2m^2 \equiv_{2m} 2m $.

\end{enumerate}

It only remains to show that order preservation is obtained, without loss of generality, taking $1\leq x \leq m$ the usual order in integers when multiplied by $-1$ and adding $2m$ to each term, we obtain $2m-1\geq 2m - x \geq m$. Finally, the inequality is given by applying equivalence classes.

\end{proof}

\section{Finite Euler coprimes}\label{secc3}

In this section we will use the set defined by Euler to build the function $\phi$ that counts the coprimes, less than a natural number. This given set with a special relationship turns out to be a partially ordered set with a minimum, and with atomic or prime elements. In addition, the function $\phi$ and its most important properties that were used in the construction of an algorithm will be remembered. Such an algorithm is implemented in Matlab2020b, as is presented in the last section. It calculates all the primes that satisfy the conjecture when the user enters a number $2m\geq 4$.

\begin{definition}\label{def8}
Let $m \in \mathbb{N}^{*}$, be fixed. The set of coprimes less than $m$ will be denoted by $\gamma(m)$. Then, it is had to: $\gamma(m)= \{s \in \mathbb{N} : s \wedge m = 1, s < m \}$. It defines the function $\phi(m)$ as the cardinal of $\gamma(m)$. Furthermore, the isomorphism between the posets $\langle \gamma, \leq\rangle$ and $\langle \Gamma(m), \preceq \rangle$ is given by: $\theta: \gamma(m)\longrightarrow \Gamma (m)$ defined by $\theta(s)= \overline{s}$. For this reason, we have that $\phi(m) = \vert \gamma(m)\vert = \vert \Gamma(m)\vert$.
\end{definition}

The basic properties of the $\phi$ function that optimize computing time are indicated.

\begin{theorem}\label{teo10}
The main properties used in the algorithm proposed in this work are presented.
\begin{enumerate}

\item $\phi(1) = 1$.

\item $\phi(p) = p - 1$ if and only if $p$ is prime.

\item If $m$ and $n$ are coprime then, $\phi(m\cdot n) =\phi(m)\cdot \phi(n)$.

\item If $m = \prod^{\nu}_{j=1}p_{j}^{n_j}, p_{j} \in \mathbb{P}, n_j, \nu \in \mathbb{N^{*}}$ is the prime factorization of $m$, except for the order of the factors, then $m = \prod^{\nu}_{j=1}p_{j}^{n_j-1}(p_{j} - 1)$.

\item  If $m$ and $n$ are nonzero then $\phi(m\cdot n) \geq \phi(m)\cdot \phi(n)$.

\end{enumerate}

\end{theorem}

The following lemma allows us to guarantee the existence of primes in $\gamma(m)$.

\begin{lemma}\label{lema1}
$\forall m>2 \exists p \in \mathbb{P} \left( p \in \gamma(m)\right) $.
\end{lemma}

\begin{proof}
\textbf{Reductio ad absurdum} suppose that there exists $m'>2$, such that, for every prime $p$, we have that $p \not\in \gamma(m')$. This implies that, for any prime $p$, one has that $p \wedge m' \neq 1$ or $p>m'$. This means that, all primes less than $m'$ divide it, then $\phi(m')=1$, absurd, so $\forall m>2 \left( \phi(m)>1\right)$.
\end{proof}

In figure \ref{fig2}, it can be seen that the pairs that satisfy the conjecture are in the set $\gamma(2m) \times \gamma(2m)$ or the pairs of the classes that satisfy the conjecture are in the quotient set $\Gamma(2m)\times \Gamma(2m)$.

\begin{figure}[tb]
\includegraphics[scale=0.7]{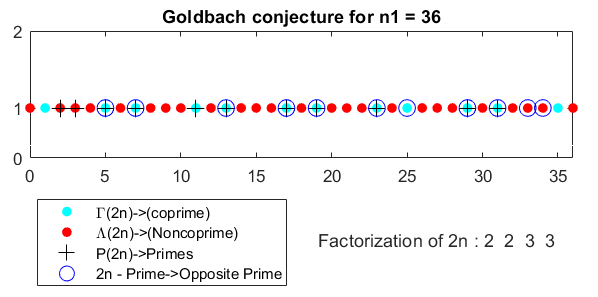}
\caption{$2m = 36$ in blue are coprime and non-coprime in red, those that satisfy the conjecture have a circle with an cross}\label{fig2}
\end{figure}

An algorithm is presented that calculates the pairs of primes $p,q \in \mathbb{P}$, that satisfy Goldbach's theorem for any $2m > 4$ pair, using Euler's $\phi(2m)$  function. In figure \ref{fig3}, we can see the coprimes in blue, that are in the lower part of the line. The non-coprimes, also shown with magenta color are always on top like in a kind of arc.

\begin{algorithm}
\small
\caption{Search algorithm for prime pairs that satisfy the conjecture}\label{alg1}
\begin{flushleft} 
$2m \gets input(\text{"Enter an even number:  "})$ \\ 
$Meet := $ Meet a couple of primes that satisfy the conjecture for $2m$\\

$obj1 \gets primes(2m)$\\
$Meet \gets $ $\emptyset$\\

For $i = 1: length(obj1)$\\
\, \, if $obj1(i) \wedge 2m - obj1(i))==1$\\
\, \, \, \, if $\phi(2m - obj1(i)) == 2m - obj1(i) -1$\\
\, \, \, \, \, \, $Meet = [Meet ;[obj1(i),2n - obj1(i);2n - obj1(i),obj1(i)]]$\\
\, \, if $obj1(i) - 1 == \phi(2n - obj1(i))$\\ 
\, \, \, \, $Meet = [Meet ;[obj1(i),2n - obj1(i);2n - obj1(i),obj1(i)]]$\\
Return $Meet$ 
\end{flushleft} 
\end{algorithm}

\begin{example}\label{ejemp2}

Given $2m = 36$, calculating $\Gamma (2m)$, the set of coprimes of $2m$ and the result of operating $\odot$ on $\Gamma(2m)$ is obtained, validating the guaranteed properties for any pair $2m$ with $m>1$.

\begin{enumerate}

\item $\Gamma(36) = \left\lbrace \overline{1},\overline{5},\overline{7},\overline{11},\overline{13},\overline{17},\overline{19},\overline{23},\overline{25},\overline{29},\overline{31},\overline{35}\right\rbrace$.

\item It is observed $\odot$ in table\ref{tabla1} of Cayley.

	\begin{table}[ht]
		\centering
		\begin{tabular}{|>{\columncolor{gray!30}}c|>{\columncolor{blue!30}}c|c|c|c|c|>{\columncolor{blue!30}}c|>{\columncolor{magenta!30}}c|c|c|c|c|>{\columncolor{magenta!30}}c|}
			\hline
			\rowcolor{gray!30}
			\cellcolor{white}{$\odot$} & $\overline{1}$ & $\overline{5}$ & $\overline{7}$ & $\overline{11}$ & $\overline{13}$ & $\overline{17}$ & $\overline{19}$ & $\overline{23}$ & $\overline{25}$ & $\overline{29}$ & $\overline{31}$ & $\overline{35}$\\
			\hline
			\rowcolor{blue!30}
			\cellcolor{gray!30}{$\overline{1}$} & \cellcolor{green!20}{$\overline{1}$} & $\overline{5}$ & $\overline{7}$ & $\overline{11}$ & $\overline{13}$ & $\overline{17}$ & $\overline{19}$ & $\overline{23}$ & $\overline{25}$ & $\overline{29}$ & $\overline{31}$ & \cellcolor{orange!20}{$\overline{35}$} \\
			\hline
			$\overline{5}$ & $\overline{5}$ & $\overline{25}$ & \cellcolor{orange!20}{$\overline{35}$} & $\overline{19}$ & $\overline{29}$ & $\overline{13}$ & $\overline{23}$ & $\overline{7}$ & $\overline{17}$ & \cellcolor{green!20}{$\overline{1}$} & $\overline{11}$ & $\overline{31}$\\
			\hline
			$\overline{7}$ & $\overline{7}$ & \cellcolor{orange!20}{$\overline{35}$} & $\overline{13}$ & $\overline{5}$ & $\overline{19}$ & $\overline{11}$ & $\overline{25}$ & $\overline{17}$ & $\overline{31}$ & $\overline{23}$ & \cellcolor{green!20}{$\overline{1}$} & $\overline{29}$\\
			\hline
			$\overline{11}$ & $\overline{11}$ & $\overline{19}$ & $\overline{5}$ & $\overline{13}$ & \cellcolor{orange!20}{$\overline{35}$} & $\overline{7}$ & $\overline{29}$ & \cellcolor{green!20}{$\overline{1}$} & $\overline{23}$ & $\overline{31}$ & $\overline{17}$ & $\overline{25}$\\
			\hline
			$\overline{13}$ & $\overline{13}$ & $\overline{29}$ & $\overline{19}$ & \cellcolor{orange!20}{$\overline{35}$} & $\overline{25}$ & $\overline{5}$ & $\overline{31}$ & $\overline{11}$ & \cellcolor{green!20}{$\overline{1}$} & $\overline{17}$ & $\overline{7}$ & $\overline{23}$\\
			\hline
			\rowcolor{blue!30}
			\cellcolor{gray!30}{$\overline{17}$} & $\overline{17}$ & $\overline{13}$ & $\overline{11}$ & $\overline{7}$ & $\overline{5}$ & \cellcolor{green!20}{$\overline{1}$} & \cellcolor{orange!20}{$\overline{35}$} & $\overline{31}$ & $\overline{29}$ & $\overline{25}$ & $\overline{23}$ & $\overline{19}$\\
			\hline
			\rowcolor{magenta!30}
			\cellcolor{gray!30}{$\overline{19}$} & $\overline{19}$ & $\overline{23}$ & $\overline{25}$ & $\overline{29}$ & $\overline{31}$ & \cellcolor{orange!20}{$\overline{35}$} & \cellcolor{green!20}{$\overline{1}$} & $\overline{5}$ & $\overline{7}$ & $\overline{11}$ & $\overline{13}$ & $\overline{17}$\\
			\hline
			$\overline{23}$ & $\overline{23}$ & $\overline{7}$ & $\overline{17}$ & \cellcolor{green!20}{$\overline{1}$} & $\overline{11}$ & $\overline{31}$ & $\overline{5}$ & $\overline{25}$ & \cellcolor{orange!20}{$\overline{35}$} & $\overline{19}$ & $\overline{29}$ & $\overline{13}$\\
			\hline
			$\overline{25}$ & $\overline{25}$ & $\overline{17}$ & $\overline{31}$ & $\overline{23}$ & \cellcolor{green!20}{$\overline{1}$} & $\overline{29}$ & $\overline{7}$ & \cellcolor{orange!20}{$\overline{35}$} & $\overline{13}$ & $\overline{5}$ & $\overline{19}$ & $\overline{11}$\\
			\hline
			$\overline{29}$ & $\overline{29}$ & \cellcolor{green!20}{$\overline{1}$} & $\overline{23}$ & $\overline{31}$ & $\overline{17}$ & $\overline{25}$ & $\overline{11}$ & $\overline{19}$ & $\overline{5}$ & $\overline{13}$ & \cellcolor{orange!20}{$\overline{35}$} & $\overline{7}$\\
			\hline
			$\overline{31}$ & $\overline{31}$ & $\overline{11}$ & \cellcolor{green!20}{$\overline{1}$} & $\overline{17}$ & $\overline{7}$ & $\overline{23}$ & $\overline{13}$ & $\overline{29}$ & $\overline{19}$ & \cellcolor{orange!20}{$\overline{35}$} & $\overline{25}$ & $\overline{5}$\\
			\hline
			\rowcolor{magenta!30}
			\cellcolor{gray!30}{$\overline{35}$} & \cellcolor{orange!20}{$\overline{35}$} & $\overline{31}$ & $\overline{29}$ & $\overline{25}$ & $\overline{23}$ & $\overline{19}$ & $\overline{17}$ & $\overline{13}$ & $\overline{11}$ & $\overline{7}$ & $\overline{5}$ & \cellcolor{green!20}{$\overline{1}$}\\
			\hline
		\end{tabular}
		\caption{Table of Cayley of $\odot$ on $\Gamma(36)$}\label{tabla1}
	\end{table} 

\end{enumerate}
\end{example}

\begin{remark}\label{remar4}

Table \ref{tabla1}, of Cayley shows the operation $\odot$ on $\Gamma ( 2m = 36)$, in addition to having beautiful symmetries, the following properties are observed:

\begin{enumerate}
\item $\odot$ is therefore closed, the results of the operation at $\overline{a} \odot \overline{b} \in \Gamma(2m)$ for each $\overline{a}, \overline{b} \in \Gamma(2m)$.

\item $\odot$ is commutative only observing the equality between the upper triangular submatrix and the lower triangular submatrix.

\item $\odot$ is modulative and the module is $\overline{1}$. The blue the rows columns can verify this property.

\item In $\odot$ not only the co-opposite elements (orange cells), but also The module (green cells) are symmetry.

\item $\odot$ is invertive, because the green cells guarantee the pair of inverses whose product is modulo $\overline{1}$.

\item The magenta y column guarantees property (b) of Theorem \ref{teo5}, since operating any element $\overline{x} \in \Gamma(2m)$ with $\overline{2m - 1}$ through $\odot$ gives its co-opposite $\overline{2m - x}$ as a result. Furthermore, while the elements of the module $\overline{1}$ column increase, the elements of the $\overline{2m - 1}$ column decrease. Just to fix ideas, taking the element $x = \overline{19}$ and operating it with $\overline{2m - 1}$, it results in $\overline{17}$, which is the co-opposite of $\overline{19}$, since $\overline{2m - 19} = \overline{17}$.
\item In the dark blue and magenta rows and columns, it can be seen that property (c) of Theorem \ref{teo5} holds; for example, when taking $a_{m} = \overline{17}$ and operating it with  $\overline{2m - 1}$ through $\odot$, it gives its co-opposite $2m - a_{m} = \overline{19} = a_{m+1}$ as a result.

\end{enumerate}
     
\end{remark}

Let's observe in the example that, the $(\overline{5},\overline{31}),(\overline{7},\overline{29}),(\overline{13},\overline{23}),(\overline{17},\overline{19})$ couples satisfy the conjecture and these are found in the blue and magenta columns. There are also the $(\overline{1},\overline{35}),(\overline{11},\overline{25})$ couples that, although their sum is even, do not satisfy the conjecture, because one of the coordinates is not a prime number. Being the commutative operation symmetric pairs that also satisfy the conjecture need not be analyzed; then, only half the elements of $\Gamma(2m)$ need to be traversed. Another present relationship is that, in the blue column, going through the coprimes smaller than $m$, an ascending order is observed; nevertheless, in the magenta column, the co-opposites are ordered in a decreasing manner, but the coprimes greater than $m$ are located there. This feature is important for proving the central outcome.

In the next section, images are presented that show the multiple symmetries, that are obtained in the operation $\odot$; furthermore, when searching for the cyclical subgroups of the $\langle \Gamma(2m) , \odot , 1 \rangle $, group of the present work, some patterns are observed that should be further explored.

\section{The beauty of co-primes}

The algorithm was implemented in Matlab 2020b; in addition, the operation was programmed for $\odot$ over $\Gamma(2m)$. The user can enter the desired $2m$ pair, with the aim of guaranteeing the Abelian group structure. The positions where the operation results in the co-opposite $2m - 1$ of module $1$.

% La opción ruled genera lineas en los ciclos anidados
% vlined crea linas verticals 
% One language pone todo en un solo idioma
% Resetcount hace que la numeración del algorimo se resetea en el cambio de sección 
% algosection es para que realice la numeración por secciones (se puede poner algochapter y cuenta por capitulos)

Different symmetries are observed when taking each element of $\Gamma(2m)$, which are the result of applying properties proving on congruences such as distributivity and $2m \equiv 0$   $(mod \, 2m)$.

However, some properties are presented that muestran these symmetries in $\Gamma(2m)$, which can be generalized to the congruencies $mod \, m$ of the $\mathbb{Z}$ in general.

\begin{theorem}\label{teo13}
Let $m \in \mathbb{N}^+$ and lets $\overline{x}, \overline{y} \in \Gamma(2m)$, be arbitrary, such that:
\begin{enumerate}

\item  If $ m \equiv 0 (mod 2)$, then we have:

\begin{itemize}
\item \textbf{Symmetry A:}  $\overline{x} \odot \overline{y} = \left(\overline{2m -x}\right)\odot \left(\overline{2m -y}\right)$

\item \textbf{Symmetry B:}  $\overline{x} \odot \overline{y} = \left(\overline{m -x}\right)\odot \left(\overline{m -y}\right)$

\item \textbf{Symmetry C:}  $\overline{x} \odot \overline{y} = \left(\overline{m +x}\right)\odot \left(\overline{m +y}\right)$

\end{itemize}

\item  If $ m \equiv 0 (mod 2)$, then we have:

\begin{itemize}

\item \textbf{Symmetry A:}  $\overline{x} \odot \overline{y} = \left(\overline{2m -x}\right)\odot \left(\overline{2m -y}\right)$

\item \textbf{Symmetry D:}  $\overline{x} \odot \overline{y} = \left(\overline{m - x}\right)\odot \left(\overline{m -y}\right) - \overline{m}$

\item \textbf{Symmetry E:}  $\overline{x} \odot \overline{y} = \left(\overline{m + x}\right)\odot \left(\overline{m +y}\right) + \overline{m}$ 

\end{itemize}

\end{enumerate}

\end{theorem}

%Figure \ref{fig4a}, for example, shows the symmetries of the operation for $\Gamma(2m = 296)$ and the $m = 148$ is even with the axes of symmetries are shown.

%\begin{figure}[tb]
%\includegraphics[scale=0.5]{Gamma256A}
%\caption{Feature surface shown the table Cayley of $\odot$ acting on $\Gamma(2m = 296)$ and the symmetries present}\label{fig4}
%\end{figure}

Figure \ref{fig4}, for example, shows the symmetries of the operation for $\Gamma(2m = 296)$ and the $m = 148$, and the symmetries of the operation for $\Gamma(2m = 450)$ and the $m = 225$ is odd with the axes of symmetries are shown.

\begin{figure}[tb]
\includegraphics[scale=0.32]{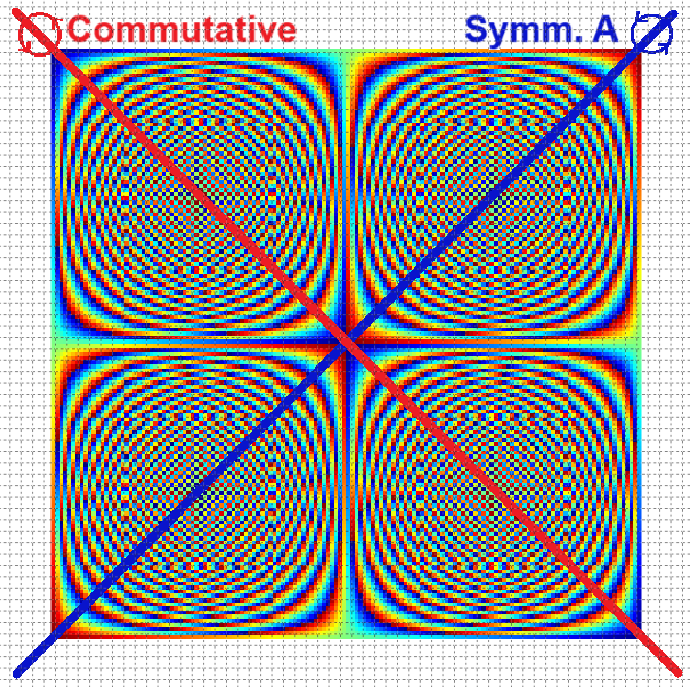}
\includegraphics[scale=0.4]{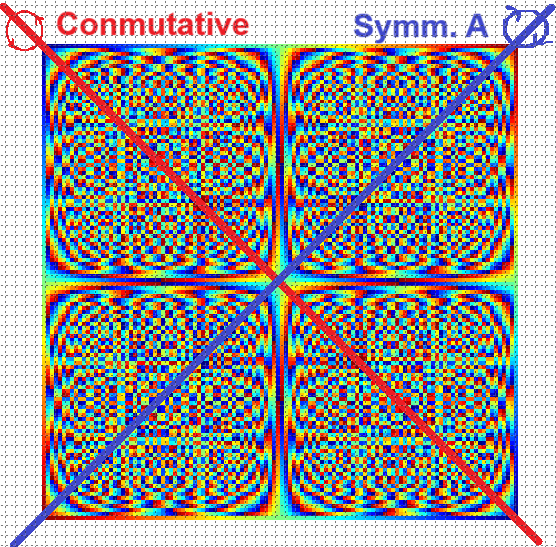}
\caption{The Cayley tables of $\odot$ acting on $\Gamma(2m = 296)$ are shown in the image on the left, and the $\Gamma(2m = 450)$, in the image on the right.}\label{fig4}
\end{figure}

Figure \ref{fig4b} shows how symmetries can be used to reach elements of higher levels, if $\Gamma(m) \subseteq \Gamma(2m) \subseteq \Gamma(2^2m) \subseteq \cdots $.

\begin{figure}[tb]
\includegraphics[scale=0.6]{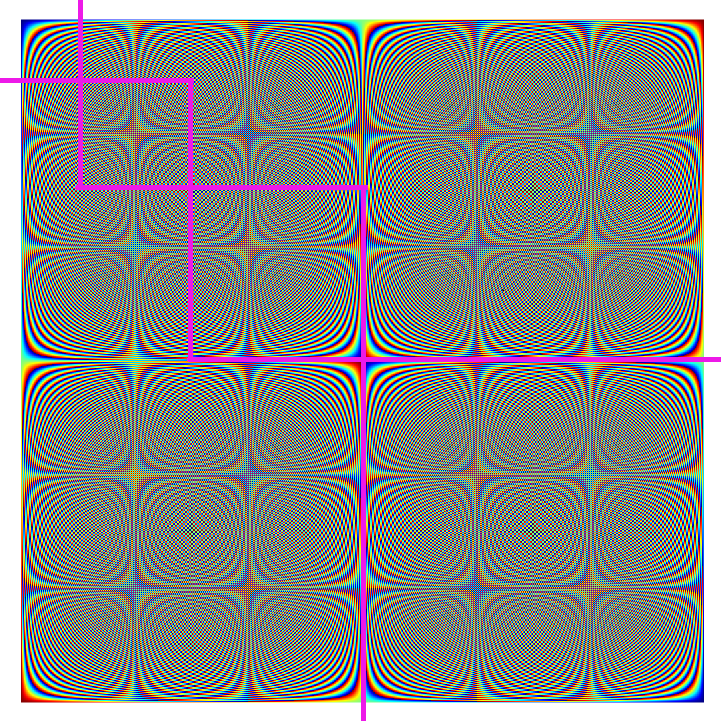}
\caption{The Cayley tables of $\odot$ acting on $\Gamma(2m = 256),\Gamma(2m = 512)$,$\Gamma(2m = 1024)$,$\Gamma(2m = 2048)$ The magenta lines show the fit.}\label{fig4b}
\end{figure}

By constructing the function $\Theta(z) = \overline{2m - z}$ and composing it we can pass $\Gamma(m)\longrightarrow \Gamma(2m)\longrightarrow \Gamma(2^2m)\longrightarrow \cdots \text{etc}$, the idea is to translate pairs that satisfy the property from the level $m$ to the level $2m$ and in turn to $2^2m$ and so on. But, you can also lower the property from high levels to lower levels.

The different cyclical subgroups were calculated for the example of $\Gamma(2m = 296)$. There are cyclical subgroups of order: $1, 2, 3, 4, 6, 9, 12, 18$ and $36$. The patterns they generate showed some elements of these groups. In the table \ref{tabla1} they are painted in different colors, according to the order of the generating element. Locating the result in figures \ref{fig5}-\ref{fig6}, this exercise can be identified.

\begin{figure}[tb]
\includegraphics[scale=0.3]{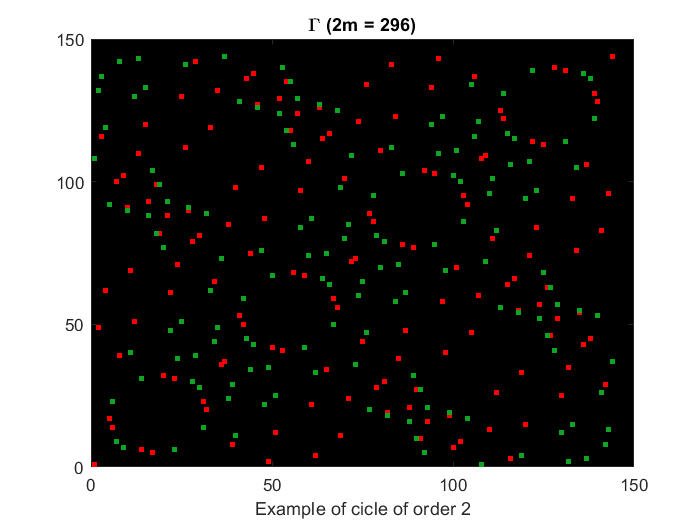}
\includegraphics[scale=0.3]{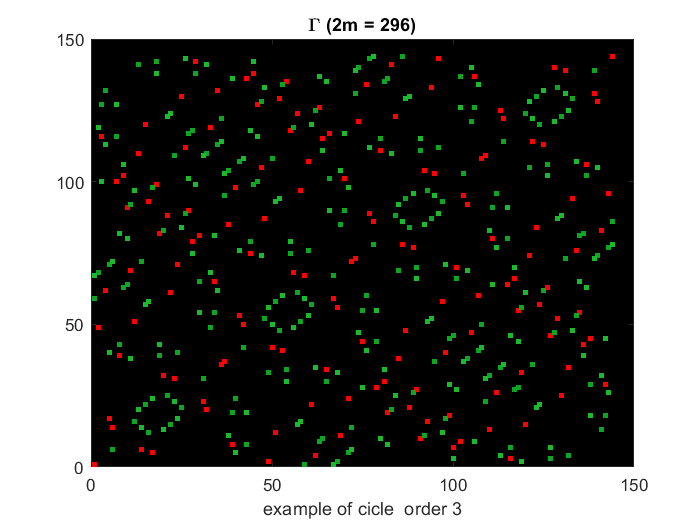}
\includegraphics[scale=0.3]{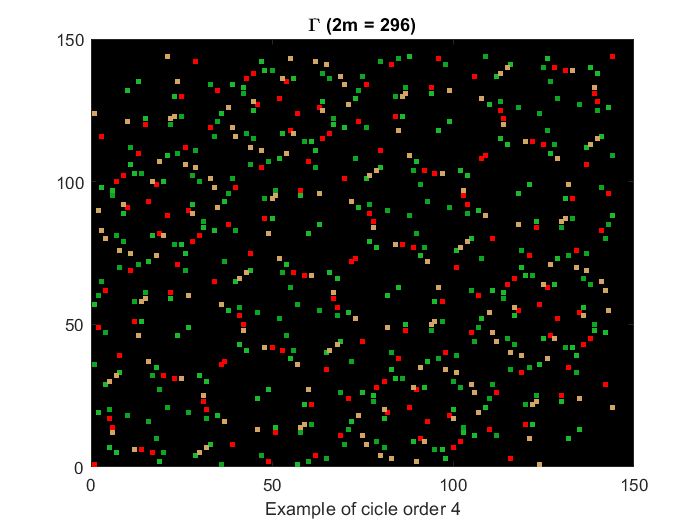}
\includegraphics[scale=0.3]{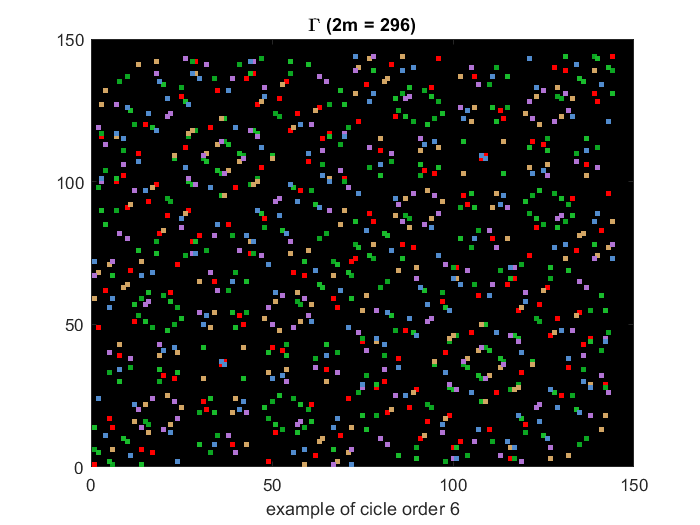}
\includegraphics[scale=0.3]{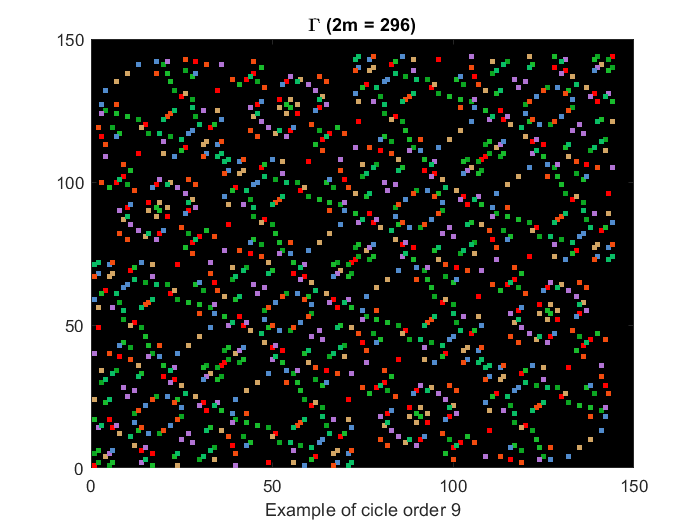}
\includegraphics[scale=0.3]{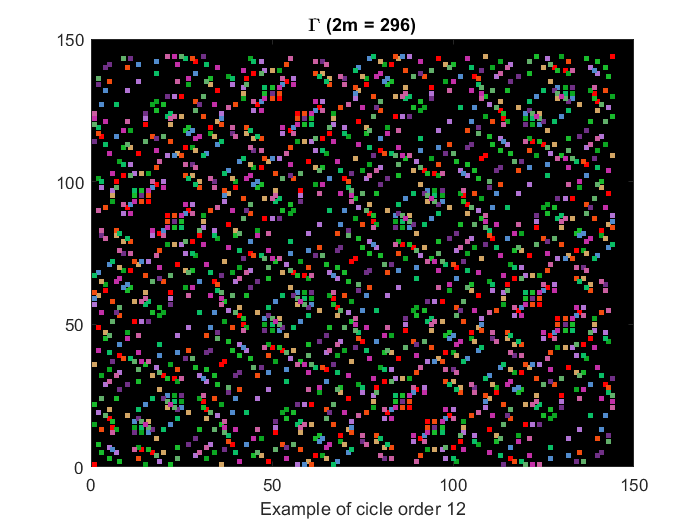}
\includegraphics[scale=0.3]{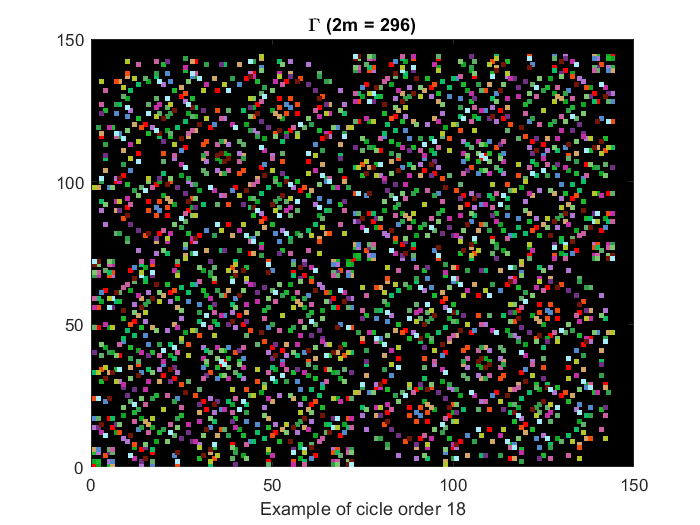}
\includegraphics[scale=0.3]{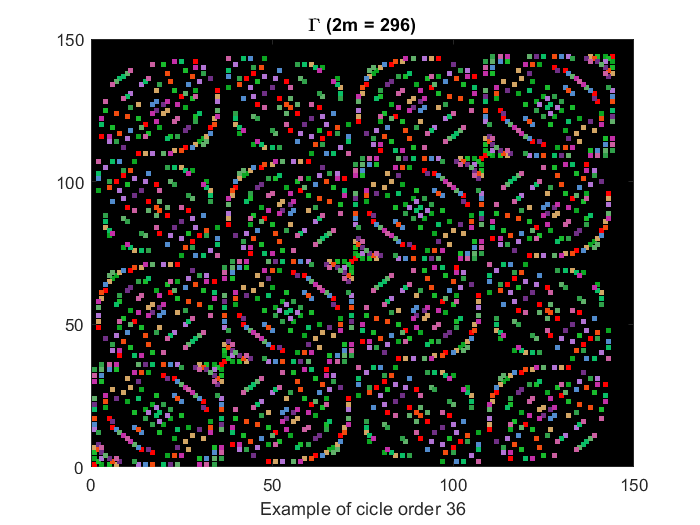}

\caption{Patterns of the cycles of order 2,3,4,6,9,12,18,36 for the structure $\langle \Gamma(2m = 296) , \odot , 1 \rangle$ }\label{fig5}
\end{figure}

The following images show the symmetries, patterns present in the orbits of cyclic subgroups for $\Gamma(2m = 180)$.

\begin{figure}[tb]
\includegraphics[scale=0.3]{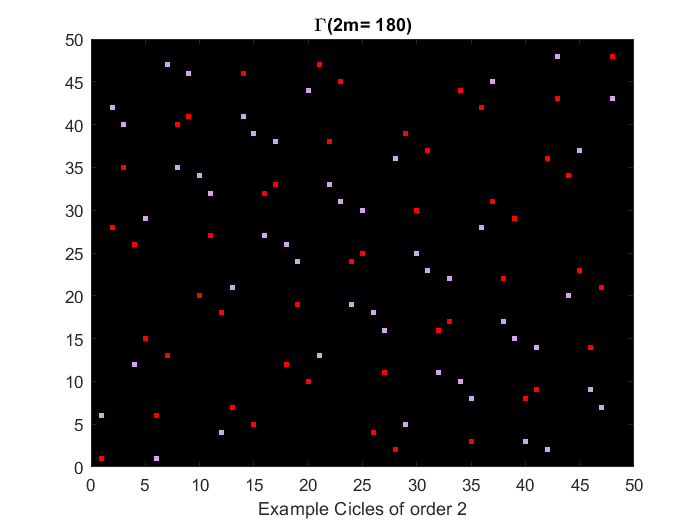}
\includegraphics[scale=0.3]{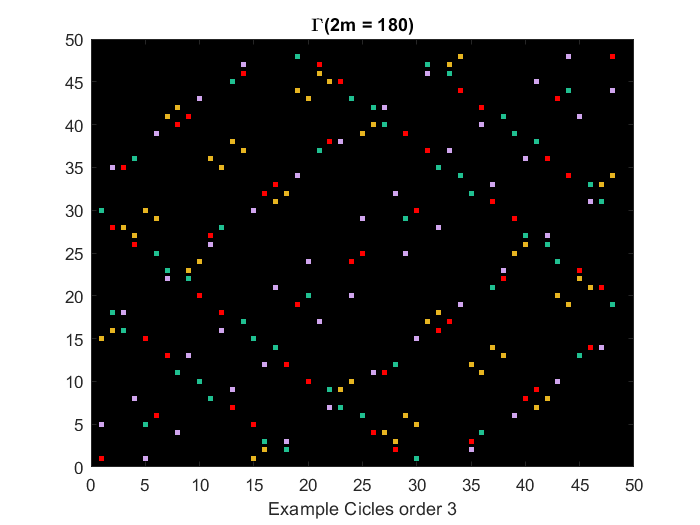}
\includegraphics[scale=0.3]{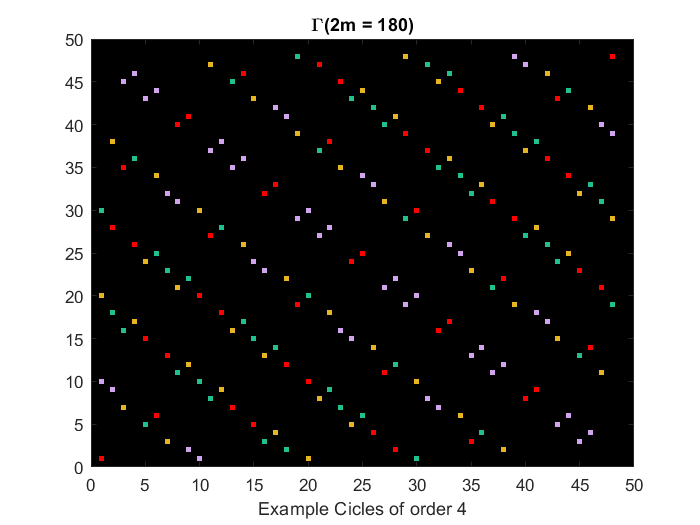}
\includegraphics[scale=0.3]{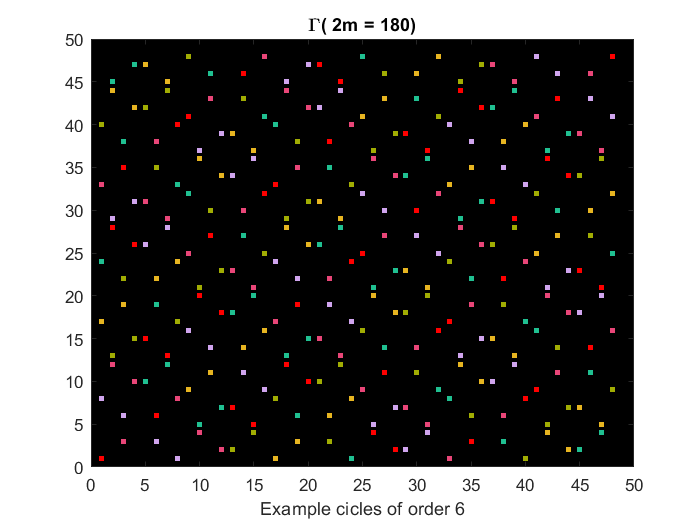}
\includegraphics[scale=0.3]{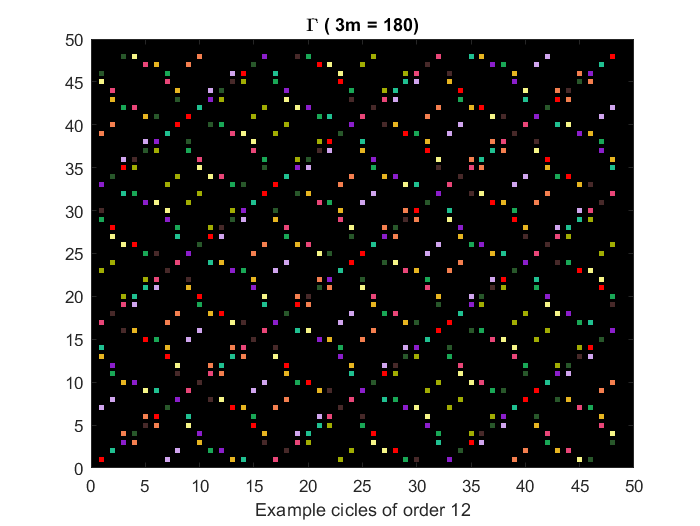}
\includegraphics[scale=0.3]{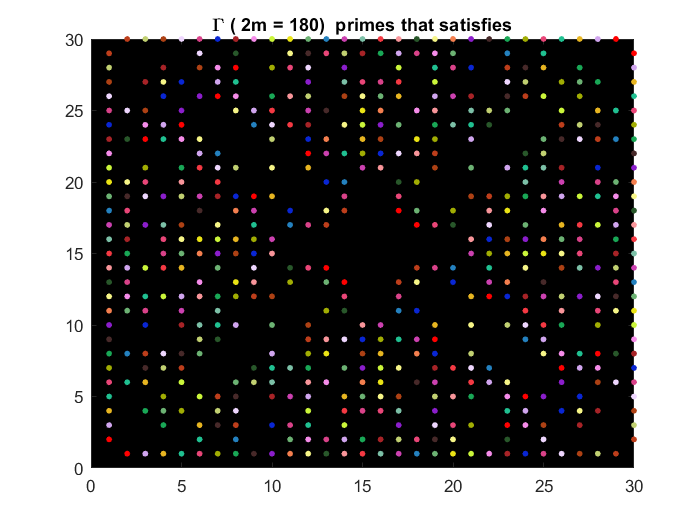}
\caption{Patterns of the cycles of order $2,3,4,6,12$ and the primes that satisfy the conjecture for the structure $\langle \Gamma(2m = 180) , \odot , 1 \rangle$ }\label{fig6}
\end{figure}

The following section presents the most powerful results of this research, which thanks to the identified symmetries allows us to visualize that the conjectures may be true.

\section{Proof of Goldbach's Conjecture}

The Theorem \ref{teo11} is presented in the section, which is the main result of this investigation. When reviewing the possible regions where the conjecture is fulfilled, one must search in the line $x + y = 2m$, when solving for $y$ as a function of $x$. The search is refined when $x$ is prime, leading to dealing with primes between $1$ and $2m$; for $x$ and, for $y$, their values vary between $1$ and $2m$. It should select those and values that are prime. 

Figure \ref{fig3} shows all the co-primes with blue for the number $2m$ and the non-coprimes in magenta. The couples that comply with a circle and a cross are also marked if the number is also the co-opposite of $x$. 

\begin{figure}[tb]
\includegraphics[scale=0.5]{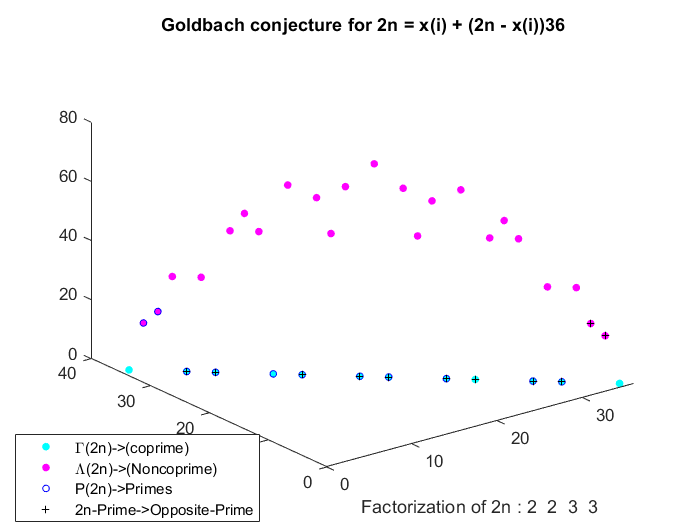}
\caption{In $2m = 36$ the coprimes are in blue and the non-coprimes are in magenta. Those that satisfy the conjecture have a circle with a cross}\label{fig3}
\end{figure}

In the previous section, the operation $\odot$ on $\Gamma (2m)$ was built. Some properties that are necessary for the proof of the conjecture were demonstrated.

\begin{theorem}[Goldbach's Conjecture]\label{teo11}
$\forall m > 2 \exists p , q \in \mathbb{P} \left( p + q = 2m \right)$, that is, "for all natural $m \geq 2$, there exist primes $p, q$, such that $p + q = 2m$". 
\end{theorem} 

\begin{proof}
\textbf{Reductio ad absurdum} Let's suppose that, there is $2m' > 4$, such that, for all $p, q \in \mathbb{P}$, we have to have $p + q \neq 2m'$ , which is equivalent to the fact that the ordered pairs $(p,q)\in \mathbb{P} \times \mathbb{P}$ are not on the line $x + y = 2m'$, solving for $y$ in terms of $x$, one has that of all the pairs $(x, 2m' -x)$ that are on the line, one of the two coordinates must not be a prime number. By constructing the group $\left\langle \Gamma(2m'), \odot, \overline{1}\right\rangle $ for $2m'>2$ by theorem \ref{teo4}, where are realized in the Cayley table the operation $\odot$. In the columns $\overline{1}$ and $\overline{2m' -1}$, find all the pairs of coprimes that are in line $x + y = 2m '$ and which are candidates to be the pairs that satisfy the conjecture. On the other hand, the Lemma \ref{lema1} guarantees the existence of primes in $\gamma(m')$, therefore, taking all the primes $x \in \gamma(2m')$, such that, we have that $1 \leq x \leq m'$ are coprime with $m'$, and their co-opposites $2m' - x \in \gamma(2m')$ must not be prime by hypothesis, Thus, all $2m ' - x$ between $m'$ and $2m'-1$ must be composite, this implies that for each $2m'-x$ in this region, there exist $s,t \in \gamma(2m')$ with $1<s <2m'$ and $1<t <2m'$ such that $2m' - x = s\cdot t$, applying the isomorphism $\theta$ by definition \ref{def8}, we have that, $\overline {m'} \preceq \overline{2m' - x}=\overline{s\cdot t} \preceq \overline{2m'-1}$. By theorem \ref{teo8} operating each term of the last inequality by $\overline{2m' - 1}$ through $\odot$, we have that, $\overline{1} \preceq \overline{x} = \overline{2m' -st} \preceq \overline{m'}$, which is equivalent by definition \ref{def6} to the inequality $\overline{1} \preceq \overline{x}= \overline{2m ' -s}\odot \overline{t} \preceq \overline{m'}$. Applying now the isomorphism $\theta^{-1}$ of poset, we have that $1 \leq x=(2m' -s)\cdot t \leq m'$, which means that all these coprimes $x$ between $1$ and $m'$ are composite numbers in $\gamma(m)$, absurd since it contradicts the existence of primes in this area, this is contradicting the Lemma \ref{lema1}.
\end{proof}

As a direct consequence of the proof of the previous theorem we have.

\begin{corollary}\label{coro1}
$\forall m > 2\exists p \in \mathbb{P}\left(  m < p < 2m\right) $.
\end{corollary}

\section{Proof of twins primes conjecture}

The purpose of this section is to demonstrate the twin primes conjecture, to visualize the ideas is presented in the following Table \ref{tabla2}, shows the operation $\odot$ on $\Gamma ( 2m = 26)$; in addition to having beautiful symmetries, the following properties are observed

\begin{definition}\label{def9}
Given $p, q \in \mathbb{P}$, it is said that $p,q$ are twins primes, if and only if, $\vert p - q \vert = 2$.
We will notice for $\mathbb{F} = \left\lbrace (p,q) \in \mathbb{P}\times \mathbb{P} : {p,q} \text{is a pair of twins primes} \right \rbrace$ to the set formed by pairs of twins prime.
\end{definition}

\begin{remark}\label{remark5}
If $p \in \mathbb{P}$ then $\phi(2p)=\phi(p) = p-1 $, when explicitly calculating $\gamma(2p)$, it is observed that they are the odd numbers less than $2p$, but for there to be at least the first pair of twin primes between $1$ and $p$, $p\geq 7$ must be. Note this result in the following example $\phi(2p=26)=\phi(p=13)=12$, in detail, $\gamma(2p=26)=\left\lbrace 1,3,5, 7,9,11,13,15,17,19,21,23,25 \right\rbrace $
\end{remark}

\begin{example}\label{ejemp3}
Given $2m = 26$, calculating $\Gamma (2m)$, the set of coprimes of $2m$, and the result of operating $\odot$ on $\Gamma(2m)$, is obtained, is the reflection that the properties validate for any pair $2m$, with $m \geq 7$, proved in the previous section:
\begin{enumerate}

\item $\Gamma(26) = \left\lbrace \overline{1},\overline{3},\overline{5},\overline{7},\overline{9},\overline{11},\overline{15},\overline{17},\overline{19},\overline{21},\overline{23},\overline{25}\right\rbrace$.

\item It is observed $\odot$ in table\ref{tabla2}.

	\begin{table}[ht]
		\centering
		\begin{tabular}{|>{\columncolor{gray!30}}c|>{\columncolor{blue!30}}c|c|c|c|c|>{\columncolor{blue!30}}c|>{\columncolor{magenta!30}}c|c|c|c|c|>{\columncolor{magenta!30}}c|}
			\hline
			\rowcolor{gray!30}
			\cellcolor{white}{$\odot$} & $\overline{1}$ & $\overline{3}$ & $\overline{5}$ & $\overline{7}$ & $\overline{9}$ & $\overline{11}$ & $\overline{15}$ & $\overline{17}$ & $\overline{19}$ & $\overline{21}$ & $\overline{23}$ & $\overline{25}$\\
			\hline
			\rowcolor{blue!30}
			\cellcolor{gray!30}{$\overline{1}$} & \cellcolor{green!20}{$\overline{1}$} & $\overline{3}$ & $\overline{5}$ & $\overline{7}$ & $\overline{9}$ & $\overline{11}$ & $\overline{15}$ & $\overline{17}$ & $\overline{19}$ & $\overline{21}$ & $\overline{23}$ & \cellcolor{orange!20}{$\overline{25}$} \\
			\hline
			$\overline{3}$ & $\overline{3}$ & $\overline{9}$ & $\overline{15}$ & $\overline{21}$ & \cellcolor{green!20}{$\overline{1}$} & $\overline{7}$ & $\overline{19}$ & \cellcolor{orange!20}{$\overline{25}$} & $\overline{5}$ & $\overline{11}$ & $\overline{17}$ & $\overline{23}$\\
			\hline
			$\overline{5}$ & $\overline{5}$ & $\overline{15}$ & \cellcolor{orange!20}{$\overline{25}$} & $\overline{9}$ & $\overline{19}$ & $\overline{3}$ & $\overline{23}$ & $\overline{7}$ & $\overline{17}$ & \cellcolor{green!20}{$\overline{1}$} & $\overline{11}$ & $\overline{21}$\\
			\hline
			$\overline{7}$ & $\overline{7}$ & $\overline{21}$ & $\overline{9}$ & $\overline{23}$ & $\overline{11}$ & \cellcolor{orange!20}{$\overline{25}$} & \cellcolor{green!20}{$\overline{1}$} & $\overline{15}$ & $\overline{3}$ & $\overline{17}$ & $\overline{5}$ & $\overline{19}$\\
			\hline
			$\overline{9}$ & $\overline{9}$ & \cellcolor{green!20}{$\overline{1}$} & $\overline{19}$ & $\overline{11}$ & $\overline{3}$ & $\overline{21}$ & $\overline{5}$ & $\overline{23}$ & $\overline{15}$ & $\overline{7}$ & \cellcolor{orange!20}{$\overline{25}$} & $\overline{17}$\\
			\hline
			\rowcolor{blue!30}
			\cellcolor{gray!30}{$\overline{11}$} & $\overline{11}$ & $\overline{7}$ & $\overline{3}$ & \cellcolor{orange!20}{$\overline{25}$} & $\overline{21}$ & $\overline{17}$ & $\overline{9}$ & $\overline{5}$ & \cellcolor{green!20}{$\overline{1}$} & $\overline{23}$ & $\overline{19}$ & $\overline{15}$\\
			\hline
			\rowcolor{magenta!30}
			\cellcolor{gray!30}{$\overline{15}$} & $\overline{15}$ & $\overline{19}$ & $\overline{23}$ & \cellcolor{green!20}{$\overline{1}$} & $\overline{5}$ & $\overline{9}$ & $\overline{17}$ & $\overline{21}$ & \cellcolor{orange!20}{$\overline{25}$} & $\overline{3}$ & $\overline{7}$ & $\overline{11}$\\
			\hline
			$\overline{17}$ & $\overline{17}$ & \cellcolor{orange!20}{$\overline{25}$} & $\overline{7}$ & $\overline{15}$ & $\overline{23}$ & $\overline{5}$ & $\overline{21}$ & $\overline{3}$ & $\overline{11}$ & $\overline{19}$ & \cellcolor{green!20}{$\overline{1}$} & $\overline{9}$\\
			\hline
			$\overline{19}$ & $\overline{19}$ & $\overline{5}$ & $\overline{17}$ & $\overline{3}$ & $\overline{15}$ & \cellcolor{green!20}{$\overline{1}$} & \cellcolor{orange!20}{$\overline{25}$} & $\overline{11}$ & $\overline{23}$ & $\overline{9}$ & $\overline{21}$ & $\overline{7}$\\
			\hline
			$\overline{21}$ & $\overline{21}$ & $\overline{11}$ & \cellcolor{green!20}{$\overline{1}$} & $\overline{17}$ & $\overline{7}$ & $\overline{23}$ & $\overline{3}$ & $\overline{19}$ & $\overline{9}$ & \cellcolor{orange!20}{$\overline{25}$} & $\overline{15}$ & $\overline{5}$\\
			\hline
			$\overline{23}$ & $\overline{23}$ & $\overline{17}$ & $\overline{11}$ & $\overline{5}$ & \cellcolor{orange!20}{$\overline{25}$} & $\overline{19}$ & $\overline{7}$ & \cellcolor{green!20}{$\overline{1}$} & $\overline{21}$ & $\overline{15}$ & $\overline{9}$ & $\overline{3}$\\
			\hline
			\rowcolor{magenta!30}
			\cellcolor{gray!30}{$\overline{25}$} & \cellcolor{orange!20}{$\overline{25}$} & $\overline{23}$ & $\overline{21}$ & $\overline{19}$ & $\overline{17}$ & $\overline{15}$ & $\overline{11}$ & $\overline{9}$ & $\overline{7}$ & $\overline{5}$ & $\overline{3}$ & \cellcolor{green!20}{$\overline{1}$}\\
			\hline
		\end{tabular}
		\caption{$\odot$ on $\Gamma(36)$}\label{tabla2}
	\end{table} 

\end{enumerate}
\end{example}

\begin{remark}\label{remar6}

Table \ref{tabla2}, shows the operation $\odot$ on $\Gamma ( 2m = 26)$; in addition to having beautiful symmetries, the following properties are observed:

\begin{enumerate}
\item $\odot$ is therefore closed, the results of the operation at $\overline{a} \odot \overline{b} \in \Gamma(2m)$, for each $\overline{a}, \overline{b} \in \Gamma(2m)$

\item The magenta y column guarantees property (b) of Theorem \ref{teo5}, since operating any element $\overline{x} \in \Gamma(2m)$ with $\overline{2m - 1}$ through $\odot$, gives its co-opposite $\overline{2m - x}$ as a result. Furthermore, while the elements of the module $\overline{1}$, column increase, the elements of the $\overline{2m - 1}$, column decrease. Just to fix ideas, taking the element $x = \overline{15}$, and operating it with $\overline{2m - 1}$, it results in $\overline{11}$, which is the co-opposite of $\overline{15}$, since $\overline{2m - 15} = \overline{11}$.
\item In the dark blue and magenta rows and columns, it can be seen that property (c) of Theorem \ref{teo5} holds; for example, when taking $a_{m} = \overline{11}$, and operating it with  $\overline{2m - 1}$ through $\odot$, it gives its co-opposite $2m - a_{m} = \overline{15} = a_{m+1}$ as a result.

\end{enumerate}
     
\end{remark}

Let's observe in the example that, the $(\overline{3},\overline{5}),(\overline{5},\overline{7}),(\overline{17},\overline{19})$, couples satisfy the conjecture and these are found in the blue and magenta columns. There are also the $(\overline{7},\overline{9}),(\overline{15},\overline{17})$, couples that, although their distance is two, do not satisfy the conjecture, because one of the coordinates is not a prime number. Being the commutative operation, symmetric pairs that also satisfy the conjecture need not be analyzed; then, only half the elements of $\Gamma(2m)$, need to be traversed. Another present relationship is that, in the blue column, going through the coprimes smaller than $m$, an ascending order is observed; nevertheless, in the magenta column, the co-opposites are ordered in a decreasing manner, but the coprimes greater than $m$, are located there. This feature is important for proving the central outcome.

\begin{theorem}[Twins primes Conjecture]\label{teo12}
There are infinitely many twins primes. That is, $ \vert \mathbb{F} \vert > \infty$, where $\mathbb{F}$, as definition \ref{def9}.
\end{theorem}

\begin{proof}
\textbf{Reductio ad absurdum} suppose that the set $\mathbb{F}$, is finite. Suppose there exists $m' \in \mathbb{P}$, such that, $m' > \max(S)$, where $S = \prod_{1}\mathbb{F} $, is the set of first few coordinates of $\mathbb{F}$, since the primes are not bounded. Then, outside of $S$, there are no twins primes, this implies that $\forall p, q \in \mathbb{P} \left( p,q \not\in S \Rightarrow \vert p - q \vert > 2\right) $, by comment \ref{remark5}, it is found that by taking $\gamma(2m')$, we have that for each pair of primes $p, q \in \mathbb{P}$ and $p, q \not\in S$, satisfying $m' < p < 2m'$ and $m' < q < 2m'$, cannot be twins primes. This implies that at least one of the consecutive coprimes $p,q$ must not be prime, also satisfying that $\vert p - q \vert > 2$, using the theory proposed, when considering the group $\langle \Gamma(2m'), \odot, 1 \rangle $, and the corresponding isomorphism $\theta$ between the posets $\left\langle \gamma(2m'),\leq \right\rangle$, and $\left\langle \Gamma(2m'),\preceq \right\rangle $, thanks to the definition \ref{def8}, applying the isomorphism $\theta$ to the previous inequalities. We obtain $\overline{m'} \prec \overline{p} \prec \overline{2m' - 1}$, and $\overline{m'} \prec \overline{q} \prec \overline{2m' - 1}$. Multiplying these inequalities by $2m' - 1$, with the operation $\odot$, by theorem\ref{teo7} to the theorem \ref{teo8}, we have that, $\overline{m'} \succ \overline{2m' - p} \succ \overline{1}$, and $\overline{m'} \succ \overline{2m' - q} \succ \overline{1}$. Now applying the isomorphism $\theta^{-1}$ to this inequality, it is obtained that $m' > 2m' - p > 1$, and $m' > 2m' - q > 1$, when we calculate the distance of these two consecutive co-opposite numbers, we have that, $\vert (2m' - p) - (2m' - q)\vert = \vert p - q \vert> 2 $, absurd, since it contradicts that for all the primes in $S$, by construction, they are also between $1$ and $m'$, their distance is exactly $2$, when they are consecutive.

\end{proof}

Finally, the most relevant conclusions of this research are indicated.

\section{Conclusions}

The notion of co-primes and non-co-primes was generalized to integers, properties were demonstrated that guarantee that $\langle \Gamma(2m),\odot, 1\rangle $ has an abelian group structure, and they identified symmetries that allow us to visualize how the conjectures can be transferred to larger groups.

Furthermore, demonstrations were made that the Goldbach and twin prime conjectures are true.

\appendix

\section{Proof of Theorems}\label{appendix} %\ref{teo1},\ref{teo2},\ref{teo3} y \ref{teo6}

\begin{theorem}\label{teo1}
For each $m>1$, it is had that $\equiv_{\mathbb{C}(m)}$ is an equivalence relation on $\mathbb{C}(m)$.
\end{theorem}

\begin{proof}
It will be proved that $\equiv_{\mathbb{C}(m)}$, satisfies the reflexive, symmetric and transitive properties.
\begin{description}

\item[a.]\textbf{Reflexive:} Taking $x \in \mathbb{C}(m)$ arbitrary, we have that $0 = 0m = x - x$ therefore, $x \equiv_{\mathbb{C}(m)} x$

\item[b.]\textbf{Symmetry:} Let $x, y \in \mathbb{C}(m)$, be arbitrary, suppose $x \equiv_{ \mathbb{C}(m)} y$ equals $\exists k \in \mathbb{Z} \left( x - y = mk \right) $ multiplying by $-1$ gives that $\exists -k \in \mathbb{Z} \left( y - x = m(-k) \right)$ which equals $y \equiv_{\mathbb{C}(m)} x$

\item[c.]\textbf{Transitive:} Let $x, y , z \in \mathbb{C}(m)$, be arbitrary, suppose $x \equiv_{\mathbb{C}(m)} y$ and $y \equiv_{\mathbb{C}(m)} z$ this equals $\exists k \in \mathbb{Z} \left(x - y = mk\right)$ and $\exists l \in \mathbb{Z} \left (z - y = ml\right)$ Solving $-y$ from the second equation and substituting it in the first equation, we have that $x + ml - z = mk$ which is equivalent to $x - z = m(k-l ) $ that is, there exists $k' = k-l \in \mathbb{Z}$ such that $x - z = mk'$ which is equivalent to $ x\equiv_{\mathbb{C}(m)} z$

\end{description}
\end{proof}

\begin{theorem}\label{teo2a}
For each $m>1$, $\cdot$, is an operation on $\mathbb{C}(m)$, which is closure, modulative, associative and commutative.
\end{theorem} 

\begin{proof}
We will prove the closure of the product since it is the least intuitive, the other properties are inherited from the product in $\mathbb{Z}$.\\ Let $x,y \in \mathbb{C}(m)$, be arbitrary, this equals
\begin{enumerate}
\item $x \wedge m = 1$, by definition.

\item $y \wedge m = 1$, by definition.

\item $\exists s_{1}, t_{1} \in \mathbb{Z} \left( s_{1}\cdot x + t_{1}\cdot m = 1\right)$, by property \ref{ecua1} over the line (1).

\item $\exists s_{2}, t_{2} \in \mathbb{Z} \left( s_{2}\cdot y + t_{2}\cdot m = 1\right)$, by property \ref{ecua1} over the line (2).

\item $\left( s_{1}\cdot x + t_{1}\cdot m \right)\cdot \left( s_{2}\cdot y + t_{2}\cdot m \right) = 1$, multiplying lines (3) and (4).

\item $ s_{1}\cdot s_{2}\cdot x\cdot y + \left( s_{1}\cdot t_{2}\cdot x+t_{1}\cdot s_{2}\cdot y +t_{1}\cdot t_{2}\cdot m \right)\cdot m = 1$, distributing and factoring.

\item $ s_{3}\cdot x\cdot y + t_{3}\cdot m = 1$ taking $s_{3}=s_{1}\cdot s_{2}$ and $t_{3}=s_ {1}\cdot t_{2}\cdot x+t_{1}\cdot s_{2}\cdot y+t_{1}\cdot t_{2}\cdot m$ from the previous line.

\item $x\cdot y \wedge m = 1$

\end{enumerate}

\end{proof}

\begin{theorem}\label{teo3a}
For each $m>1$, it is had that \\ 
$\forall x,y,u,v \in \mathbb{C}(m) \left(  x \equiv_{\mathbb{C}(m)} y \wedge u \equiv_{\mathbb{C}(m)} v \Rightarrow x \cdot u \equiv_{\mathbb{C}(m)} y\cdot v \right) $. 
\end{theorem} 

\begin{proof}
Let $x,y, u, v \in \mathbb{C}(m)$, be arbitrary, such that:
\begin{enumerate}
\item $x \equiv_{\mathbb{C}(m)} y $, by hypothesis.

\item $u \equiv_{\mathbb{C}(m)} v $, by hypothesis.

\item $\exists s \in \mathbb{Z} \left( x - y = m\cdot s\right) $, also $x \wedge m = 1$ and $y \wedge m = 1$ by definition of $\equiv_{\mathbb{C}(m)}$ in $\mathbb{C}(m)$ in line (1).

\item $\exists t \in \mathbb{Z} \left( u - v = m\cdot t \right) $, also $u \wedge m = 1$ and $v \wedge m = 1$ by definition of $\equiv_{\mathbb{C}(m)}$ in $\mathbb{C}(m)$ in line (2).

\item $\exists s \in \mathbb{Z} \left( x = m\cdot s + y\right)$, further $x \wedge m = 1$ and $y \wedge m = 1$ solving for $x $ in line (3).

\item $\exists t \in \mathbb{Z} \left( u = m\cdot t + v\right) $, further $u \wedge m = 1$ and $v \wedge m = 1$ solving for $u $ in line (4).

\item $x\cdot u = \left( m\cdot s + y \right)\cdot \left( m\cdot t + v \right)$, multiplying lines (5) and (6).

\item $ x\cdot u = m\cdot \left(m\cdot s\cdot t+s\cdot v +y\cdot t \right) + y\cdot v $, distributing and factoring.

\item $ x\cdot u - y\cdot v = m\cdot s'$, taking $s'=m\cdot s\cdot t+s\cdot v +y\cdot t$ from the previous line.

\end{enumerate}
Therefore, $x \cdot u \equiv_{\mathbb{C}(m)} y\cdot v$.
\end{proof}

\begin{theorem}\label{teo5a}
For each $m>1$, in $\left\langle \mathbb{C}(m)/\equiv_{\mathbb{C}(m)},\odot,\overline{1} \right\rangle $, it is had that $\overline{m-1}\odot\overline{m-1}=\overline{1}$.
\end{theorem}

\begin{proof}
Suppose $m>1$ then, $m-1>0$.
\begin{enumerate} 
\item $\overline{m-1}\odot\overline{m-1}=\overline{(m-1)\cdot (m-1)}$, by compatibility.

\item $\; \; \; \; \; \; \; \; \; \; \; \; \; \; \; \; \; \; \; \; \; \; \; \; =\overline{m^2 - 2m + 1}$, by distributive property.

\item $\; \; \; \; \; \; \; \; \; \; \; \; \; \; \; \; \; \; \; \; \; \; \; \; =\overline{m\cdot(m - 2) + 1}$, by factoring.

\item $m\cdot(m - 2) + 1\equiv_{\mathbb{C}(m)} 1$, then $m\cdot(m - 2)\equiv_{m} 0$.

\end{enumerate}
Since the classes are equal when their elements are related and viceversa, and by the transitivity of equality, we obtain $\overline{m-1}\odot \overline{m-1}=\overline{1}$
\end{proof} 

\begin{theorem}\label{teo6a}
For each $m>1$, it follows that $\left\langle \mathbb{C}(m)/\equiv_{\mathbb{C}(m)},\preceq \right\rangle $ is a poset. Also, $\preceq$ is total order.
\end{theorem}

\begin{proof}
The reflexive, transitive and total being properties are simple, only the antisymmetry test is done. Let $\overline{x},\overline{y}\in \mathbb{C}(m)/\equiv_{\mathbb{C}(m)}$ such that:
\begin{enumerate}
\item $\overline{x} \preceq \overline{y}$, by hypothesis.

\item $\overline{y} \preceq \overline{x}$, by hypothesis.

\item $\exists s, t \in \left\lbrace 1,2,...,m-1\right\rbrace \left(s\leq t \wedge s\equiv_{\mathbb{C}(m) } x \wedge t\equiv_{\mathbb{C}(m)} y \right)$, is equivalent to line(1) by definition of $\preceq$.

\item $\exists s, t \in \left\lbrace 1,2,...,m-1\right\rbrace \left(t\leq s\wedge t\equiv_{\mathbb{C}(m) } y \wedge s \equiv_{\mathbb{C}(m)} x \right)$, is equivalent to line(2), by definition of $\preceq$.

\item $\exists s, t \in \left\lbrace 1,2,...,m-1\right\rbrace \left(t = s\wedge t\equiv_{\mathbb{C}(m)} y \wedge s \equiv_{\mathbb{C}(m)} x \right)$, by the antisymmetry of the usual order of $\leq$ on $\left\lbrace 1,2,...,m- 1\right\rbrace$.

\item $x \equiv_{\mathbb{C}(m)} y $, by transitivity of $\equiv_{\mathbb{C}(m)}$.
\end{enumerate}
Since the classes are equal when their elements are related and viceversa, $\overline{x} = \overline{y}$. Such that, $\preceq$ it is antisymmetric.
\end{proof}

\subsection*{Acknowledgments:}
To Gonzalo Medina Arellano for his support in LaTeX, and to my wife Elisabeth for her continued support in translating this document, I also want to thank Veritasium for their work in spreading scientific curiosities and mathematical conjectures that inspired me to tackle this challenge.
  
\bibliographystyle{amsplain}

\end{document}